\documentclass[10pt]{amsart}
\usepackage{graphics,graphicx, xcolor}
\usepackage{caption}
\usepackage{amsfonts}
\usepackage{epsfig}
\usepackage{verbatim}
\usepackage{amssymb}
\usepackage{latexsym}
\usepackage{amsmath, amsthm, amssymb}

\newtheorem{corollary}{Corollary}[section]

\newtheorem{lemma}[corollary]{Lemma}

\newtheorem{theorem}[corollary]{Theorem}
\newcommand{\mylabel}[1]{\label{#1}
            \ifx\undefined\stillediting
            \else \fbox{$#1$}\fi }
\newcommand{\BE}{\begin{equation}}

\newcommand{\EEQ}{\end{equation}}
\newcommand{\rfb}[1]{\mbox{\rm
   (\ref{#1})}\ifx\undefined\stillediting\else:\fbox{$#1$}\fi}

\newfont{\Blackboard}{msbm10 scaled 1200}

\newfont{\roma}{cmr10 scaled 1200}

\def\CC{\rm \hbox{C\kern-.56em\raise.4ex
         \hbox{$\scriptscriptstyle |$}\kern+0.5 em }}

\makeatletter

\@addtoreset{equation}{section}
\makeatother

\newcommand{\mm}    {{\hbox{\hskip 0.5pt}}}

\newcommand{\bluff} {{\hbox{\raise 15pt \hbox{\mm}}}}
\makeatletter
\def\section{\@startsection {section}{1}{\z@}{-3.5ex plus -1ex minus
    -.2ex}{2.3ex plus .2ex}{\large\bf}}
\makeatother
\def\be{\begin{equation}}
\def\ee{\end{equation}}

\input{epsf}

\begin{document}
\thispagestyle{empty}
\title[Thermoelastic system with delay]{Well-posedness and exponential stability of a thermoelastic system with internal delay}
\author{Smain Moulay Khatir}
\address{UR Analysis and Control of PDEs, UR 13ES64, Department of Mathematics, Faculty of Sciences of Monastir, University of Monastir, Tunisia \\
Laboratory of Analysis and Control of Partial Differential Equations, Djillaly Liabes University, Sidi Bel Abbes, Algeria.}
\email{s.moulay\_khatir@yahoo.fr}

\author{Farhat Shel}
\address{UR Analysis and Control of PDEs, UR 13ES64, Department of Mathematics, Faculty of Sciences of Monastir, University of Monastir, Tunisia}
\email{farhat.shel@ipeit.rnu.tn} 

\begin{abstract} The presence of a delay in a thermoelastic system destroys the well-posedness and the stabilizing effect of the heat conduction \cite{Rac12}. To avoid this problem we add to the system, at the delayed equation, a Kelvin-Voigt damping. At first, we prove the well-posedness of the system by the semigroup theory. Next, under appropriate assumptions, we prove the exponential stability of the system by introducing a suitable Lyapunov functional.
\end{abstract}


\maketitle

\tableofcontents

\section{Introduction}
Let us consider the following thermoelastic system with delay
\begin{equation}
\left\{ 
\begin{tabular}{l}
$u_{tt}(x,t)-\alpha u_{xx}(x,t-\tau)+\gamma \theta _x(x,t)=0,\;\;\;\;\;\;\text{ in }(0,\ell )\times (0,\infty ),$ \\
$\theta_{t}(x,t)-\kappa\theta_{xx}(x,t)+\gamma
u_{xt}(x,t)=0,\;\;\;\;\;\;\;\;\;\;\;\;\;\text{ in }(0,\ell)\times (0,\infty ),$\\
$u(0,t)=u(\ell ,t)=\theta_x (0,t)=\theta_x(\ell ,t)=0,\;\;\;\;\;\;\;\;\;t\geq 0$
\end{tabular}
\right.  \label{d1.0.0}
\end{equation}
where $\alpha , \gamma , \kappa $ and $\ell$ are some positive constants. The functions $u=u(x,t)$ and $\theta =\theta (x,t)$ describe respectively the displacement and the temperature difference, with $x \in (0,\ell)$ and $t \geq 0.$ Moreover,
 $\tau >0$ is the time delay. Racke proved in \cite{Rac12} that, under some initial and boundary conditions, the system (\ref{d1.0.0}) is not well posed and unstable even if $\tau$ is relatively small. However, it is well known that, in the absence of delay, the damping through the heat conduction is strong enough to produce an exponential stable system (see for example \cite{Rac02, Han92, Riv92}), and specially,\cite{Rac02} and \cite{Liu99} where various types of boundary conditions are associated to the one dimensional  thermoelastic systems.
 
 In recent years, the PDEs with time delays effects become an active area of research. In fact, time delays so often arise in many applications since,  most physical phenomena not only depend on the present state but also on some past occurrences, see for instance \cite{Suh80} and references therein, but as for the classical thermoelastic system, an arbitrary small delay may destroy the well-posedness of the problem or may destroy the stability, see also \cite{ Jor08, Bat05, Dat97, Dat88}.

In order to solve the problem, additional conditions or control terms have been used, we refer to \cite{Dat86, Dat88, Nic06, Amm09', Amm15, Nic18}, see also \cite{Mas16} and references therein . In this paper we add to the delayed equation, a Kelvin-Voigt damping of the form $-\beta u_{xxt}(x,t)$ for some real positive number $\beta,$
which eventually depends on $\alpha, \gamma, \kappa$ and $\tau.$ Then our system takes the form
  \begin{equation}
\left\{ 
\begin{tabular}{l}
$u_{tt}(x,t)-\alpha u_{xx}(x,t-\tau)-\beta u_{xxt}(x,t)+\gamma \theta _x(x,t)=0,\;\;\text{ in } \Omega\times (0,\infty ),$ \\
$\theta_{t}(x,t)-\kappa\theta_{xx}(x,t)+\gamma
u_{xt}(x,t)=0,\;\;\;\;\;\;\;\;\;\;\;\;\;\;\;\;\;\;\;\;\;\;\;\;\;\;\;\;\;\;\text{ in } \Omega \times (0,\infty ),$\\
$u(0,t)=u(\ell ,t)=0,\;\;\;\;\;\;\;\;\;\;\;\;\;\;\;\;\;\;\;\;\;\;\;\;\;\;\;\;\;\;\;\;\;\;\;\;\;\;\;\;\;\;\;\;\;\;\;\;\;\;\;\;\;\;\;\;\text{ in } (0,\infty ),$ \\ 
$\theta_x(0,t)=\theta_x(\ell ,t)=0,\;\;\;\;\;\;\;\;\;\;\;\;\;\;\;\;\;\;\;\;\;\;\;\;\;\;\;\;\;\;\;\;\;\;\;\;\;\;\;\;\;\;\;\;\;\;\;\;\;\;\;\;\;\text{ in }(0,\infty ),$ \\ 
$u_{x}(x,t-\tau)=f_0(x,t-\tau),\;\;\;\;\;\;\;\;\;\;\;\;\;\;\;\;\;\;\;\;\;\;\;\;\;\;\;\;\;\;\;\;\;\;\;\;\;\;\;\;\;\;\;\;\;\;\text{ in } \Omega\times (0,\tau ),$ \\
$u(x,0)=u_0(x), u_t(x,0)=u_1(x), \theta (x,0)=\theta _0(x),\;\;\;\;\;\;\;\;\;\;\;\;\text{ in } \Omega$
\end{tabular}
\right.  \label{d1.0}
\end{equation}
where the initial data $(u_0,u_1,  f_0, \theta_0)$ belongs to a suitable space and with  $\Omega=(0,\ell).$ We meanly investigate well-posedness and exponential stability of such initial-boundary value problem.

This idea arises from \cite{Amm15'} where the authors added a Kelvin-Voigt damping term to the abstract equation. More precisely, they considered the following system 
  \begin{equation}
\left\{ 
\begin{tabular}{l}
$u^{\prime \prime}(t) +aBB^*u^{\prime}(t)+BB^*u(t-\tau)=0,\;\;\text{ in }(0,\infty ),$ \\
$u(0)=u_0,\;\;u^\prime(0)=u_1,$ \\ 
$B^*u(t-\tau)=f_0(t-\tau),\;\;\;\;\;\;\;\;\;\;\;\;\;\;\;\;\;\;\;\;\;\;\;\text{ in }  (0,\tau ),$
\end{tabular}
\right.  \label{d1.01}
\end{equation}
where a "prime" denotes a one-dimensional derivative with respect to "t" and
where $B:\mathcal{D}(B)\subset H_1\rightarrow H$ is a linear unbounded operator from a Hilbert space $H_1$ to a Hilbert space $H,$ such that $B^*,$ the adjoint of $B,$ satisfies some properties of coercivity and compact embedding.  They obtained an exponential decay result under the assumption $\tau \leq a.$

In \cite{Mus13}, the authors dropped the time delay in the harmonic term of the elastic equation in (\ref{d1.0.0}) and added a delay term of the form $\int_{\tau_1}^{\tau_2}\mu(s)\theta_{xx}(x,t-s\tau)ds$ in the heat equation, where $\tau_1$ and $\tau_2$ are non-negative constants such that $\tau_1<\tau_2$ and $\mu:[\tau_1,\tau_2]\rightarrow \mathbb{R}$ is a bounded function. They proved an exponential decay result under the condition $\int_{\tau_1}^{\tau_2}|\mu(s)|ds<\kappa.$

 We define the energy of a solution of problem (\ref{d1.0}) as
 \begin{equation*}
E(t) :=\frac{1}{2}\int_{\Omega}\left( u_{t} ^{2}(x,t)+\alpha u_{x}^{2}(x,t)+ \theta ^{2}(x,t)\right)dx+\xi\int_{\Omega}\int_{0}^{1} u_{x}^2(x,t-\tau \rho) d\rho dx
\end{equation*}
where $\xi >0$ is a parameter fixed later on.

The paper is organized as follow. In section 2 we first formulate the problem (\ref{d1.0}) into an appropriate Hilbert space, and then we study the well-posedness of the system using semigroup theory. In section 3, we prove, using Lyapunov's method, a result of exponential stability of system (\ref{d1.0}).
 \section{Well-posedness of the problem}
 We introduce, as in \cite{Amm15'}, the new variable 
 \begin{equation}
 z(x,\rho ,t)=u_x(x,t-\tau \rho),\;\;\;\;\;\;\text{ in } \Omega\times (0,1)\times (0,\infty ),
 \end{equation}
 Clearly, $z(x,\rho,t)$ satisfies
 \begin{eqnarray}
 \tau z_t(x,\rho,t)+z_{\rho}(x,\rho,t)&=&0,\;\;\;\;\;\;\;\;\;\;\;\;x\in \Omega,\;\; \rho \in (0,1),\;\; t \in (0,+\infty),\\
 z(x,0,t)&=&u_x(x,t),\;\;\;x\in  \Omega,\;\; t \in (0,+\infty).
 \end{eqnarray}
 
Then, problem (\ref{d1.0}) takes the form
\begin{eqnarray}
u_{tt}(x,t)-\alpha z_{x}(x,1,t)-\beta u_{xxt}(x,t)+\gamma \theta _x(x,t)=0,&&\text{ in } \Omega\times (0,\infty ), \label{e1}\\
\tau z_t(x,\rho,t)+z_{\rho}(x,\rho,t)=0,&&\text{ in } \Omega\times(0,1)\times (0,+\infty), \label{e2}\\
\theta_{t}(x,t)-\kappa\theta_{xx}(x,t)+\gamma
u_{xt}(x,t)=0,&&\text{ in } \Omega\times (0,\infty ), \label{e3}\\
u(0,t)=u(\ell ,t)=0,&&\text{ in } (0,\infty ), \\ 
\theta_x(0,t)=\theta_x(\ell ,t)=0,&&\text{ in }(0,\infty ), \label{e4}
\end{eqnarray}
\begin{eqnarray} 
z(x,0,t)=u_{x}(x,t),&&\text{ in } \Omega\times (0,\infty ), \label{e5}\\
u(x,0)=u_0(x), u_t(x,0)=u_1(x), \theta (x,0)=\theta _0(x),&&\text{ in } \Omega, \label{e6}\\
 z(x,\rho, 0)=f_0(x,-\tau \rho),&&\text{ in } \Omega\times (0,1 ). \label{e7}
\end{eqnarray}

Observe that it follows from (\ref{e3})-(\ref{e4}) that $\int_\Omega \theta_{t}(x,t)dx=0$ that is, $\int_\Omega \theta(x,t)dx$ is conservative all the time. Without loss of generality, we assume that $\int_\Omega \theta(x,t)dx=0$. Otherwise, we can make the substitution $\tilde{\theta}(x,t)=\theta (x,t)-\frac{1}{\ell}\int_\Omega \theta_0(x)dx,$ in fact $(u,v,z,\theta)$ and $(u,v,z,\tilde{\theta})$ satisfy the same system (\ref{e1})-(\ref{e7}).

Let 
\begin{equation*}
\mathcal{H}=\left\{(f,g,p,h)\in H^{1}_0(\Omega)\times L^2(\Omega)\times L^2(\Omega \times (0,1)) \times L^2(\Omega)\mid \int_\Omega h(x)dx=0 \right\}.
\end{equation*}
Equipped with the following inner product: for any $U_k=(f_k,g_k,p_k,h_k) \in \mathcal{H},\;\;k=1,2,$
\begin{equation*}
\left\langle U_1,U_2\right\rangle _{\mathcal{H}}=\int_\Omega \left(\alpha f_{1x}(x)f_{2x}(x)+ g_1(x)g_2(x)+h_1(x)h_2(x) \right)dx+\xi\int_\Omega \int_0^1p_1(x,\rho)p_2(x,\rho)d\rho dx, 
\end{equation*}
$\mathcal{H}$ is a Hilbert space.

Define
\begin{equation*}
U:=(u,u_t,z,\theta)
\end{equation*} 
then, problem (\ref{d1.0}) can be formulated as a first order system of the form
\begin{equation}
\left\{ 
\begin{array}{c}
U^\prime =\mathcal{A}U, \\ 
U(0)=(u_0,u_1,f_0(.,-.\tau), \theta_0)
\end{array}
\right.  \label{d500}
\end{equation}
where the operator $\mathcal{A}$ is defined by
\begin{equation*}
\mathcal{A}\left( 
\begin{array}{c}
u \\ 
v \\ 
z \\ 
\theta
\end{array}
\right) =\left( 
\begin{array}{c}
v \\ 
(\alpha z(.,1)+\beta v_{x})_x-\gamma \theta _x \\ 
-\frac{1}{\tau}z_\rho \\ 
-\gamma v_{x}+\kappa \theta _{xx} 
\end{array}
\right)
\end{equation*}
with domain
\begin{equation*}
\mathcal{D}(\mathcal{A})=\left\{ 
\begin{array}{c}
U=(u,v,z,\theta)\in \mathcal{H}\cap \left[ H^{1}_0(\Omega)\times H^{1}_0(\Omega)\times L^2(\Omega ; H^1(0,1))\times  H^2(\Omega)\right]   \mid \\ 
\;\;z(.,0)=u_x\;\;\;\text{and} \;\; 
(\alpha z(.,1)+\beta v_{x}) \in H^1(\Omega)
\end{array}
\right\}
\end{equation*}
in the Hilbert space $\mathcal{H}.$

For to establish the existence of solution, we will prove that the operator $\mathcal{A}$ generates a $\mathcal{C}_0$-semigroup, and to do this, we will prove that $\mathcal{A}-mId$ generates $\mathcal{C}_0$-semigroup (of contractions), for an appropriate real number $m,$ function of $\xi, \alpha, \beta$ and $\tau.$ Then we apply the bounded perturbation theorem (Sect. III.1 of \cite{eng99}). In fact, we begin by the following result
\begin{lemma}
If $\xi>\frac{2\tau \alpha^2}{\beta},$ then there exists $m \in \mathbb{R}$ such that $\mathcal{A}-mId$ is dissipative maximal.
\end{lemma}
\begin{proof}
Take $U=(u,v,z,h) \in \mathcal{D}(\mathcal{A}).$
\begin{eqnarray}
\left\langle \mathcal{A}U,U\right\rangle _{\mathcal{H}}&=&\alpha\int_\Omega v_{x}(x)u_{x}(x)dx+ \int_\Omega \left((\alpha z(.,1)+\beta v_{x})_x(x)-\gamma \theta _x\right) v(x)dx \notag\\
&+&\int_\Omega \left(-\gamma v_{x}+\kappa \theta _{xx}\right)(x) \theta(x)dx-\frac{\xi}{\tau}\int_\Omega \int_0^1z_\rho(x,\rho)z(x,\rho)d\rho dx. \label{dis-1}
\end{eqnarray}
Integrating by parts, using boundary conditions of $u,v$ and $\theta$ to get
\begin{eqnarray*}
 \int_\Omega \left((\alpha z(.,1)+\beta v_{x})_x-\gamma \theta _x\right)(x) v(x)dx
+\int_\Omega \left(-\gamma v_{x}+\kappa \theta _{xx}\right)(x) \theta(x)dx \\
=-\alpha \int_\Omega  z(x,1)v(x)dx-\beta\int_\Omega v^2_{x}(x)dx
-\kappa\int_\Omega \theta _{x}^2(x)dx.
\end{eqnarray*}
Integrating by parts in $\rho$, we get
\begin{equation*}
\int_\Omega \int_0^1z_\rho(x,\rho)z(x,\rho)d\rho dx=\frac{1}{2}\int_\Omega  \left( z^2(x,1)-z^2(x,0)\right)dx.
\end{equation*}
Then (\ref{dis-1}) become
\begin{eqnarray*}
\left\langle \mathcal{A}U,U\right\rangle _{\mathcal{H}}&=&\alpha\int_\Omega v_{x}(x)u_{x}(x)dx-\alpha \int_\Omega  z(x,1)v_x(x)dx\\
&-&\beta\int_\Omega v^2_{x}(x)dx
-\kappa\int_\Omega \theta _{x}^2(x)dx-\frac{\xi}{2\tau}\int_\Omega  \left( z^2(x,1)-z^2(x,0)\right)dx,
\end{eqnarray*}
from which follows, using the Young's inequality and that $z(x,0)=u_x(x),$
\begin{eqnarray*}
\left\langle \mathcal{A}U,U\right\rangle _{\mathcal{H}}&\leq & \left(\alpha\varepsilon -\beta \right) \int_\Omega v_{x}^2(x)dx+\left(\frac{\alpha}{2\varepsilon}-\frac{\xi}{2\tau} \right) \int_\Omega  z^2(x,1)dx+\left(\frac{\alpha}{2\varepsilon}+\frac{\xi}{2\tau} \right) \int_\Omega  u^2_x(x)dx\\
& -&\kappa\int_\Omega \theta _{x}^2(x)dx.
\end{eqnarray*}
Choosing $\alpha\varepsilon =\frac{\beta}{2}$, or equivalently, $\varepsilon =\frac{\beta}{2\alpha},$ we get
\begin{equation*}
\left\langle \mathcal{A}U,U\right\rangle _{\mathcal{H}}\leq  -\frac{\beta}{2} \int_\Omega v_{x}^2(x)dx+\left(\frac{\alpha^2}{\beta}-\frac{\xi}{2\tau} \right) \int_\Omega  z^2(x,1)dx+\left(\frac{\alpha^2}{\beta}+\frac{\xi}{2\tau} \right) \int_\Omega  u^2_x(x)dx-\kappa\int_\Omega \theta _{x}^2(x)dx.
\end{equation*}
Then we choose $\xi>0$ such that $\frac{\alpha^2}{\beta}-\frac{\xi}{2\tau} <0,$ that is,  $\xi >\frac{2\tau \alpha^2}{\beta}.$ Furthermore we take $m=\frac{\alpha^2}{\beta}+\frac{\xi}{2\tau}>\frac{2\alpha^2}{\beta},$ to get
\begin{equation*}
\left\langle \left( \mathcal{A}-mId\right)U,U\right\rangle _{\mathcal{H}} \leq -\frac{\beta}{2} \int_\Omega v_{x}^2(x)dx+\left(\frac{\alpha^2}{\beta}-\frac{\xi}{2\tau} \right) \int_\Omega  z^2(x,1)dx-\kappa\int_\Omega \theta _{x}^2(x)dx \leq 0
\end{equation*}
which means that the operator $\mathcal{A}-mId\;$ is dissipative. 

Now, we will prove the maximality of $\mathcal{A}-mId.\;$ It suffices to show that $\lambda Id-\mathcal{A}$ is surjective for a fixed $\lambda >m.$ Given $(f,g,p,h) \in \mathcal{H},$ we look for $U=(u,v,z,\theta) \in \mathcal{D}(\mathcal{A}),$ solution of
$$\left(\lambda Id- \mathcal{A}\right) \left( 
\begin{array}{c}
u \\ 
v \\ 
z \\ 
\theta
\end{array}
\right)=\left( 
\begin{array}{c}
f \\ 
g \\ 
p \\ 
h
\end{array}
\right),$$ 
that is verifying
\begin{equation}
\left\{ 
\begin{tabular}{l}
$\lambda u- v=f,$ \\
$\lambda v-\left(\alpha z(.,1)+\beta v_x \right)_x+\gamma \theta_x =g ,$\\
$\lambda z- \frac{1}{\tau}z_\rho =p,$ \\ 
$\lambda \theta + \gamma v_x -\kappa \theta_{xx}= h.$
\end{tabular}
\right.  \label{max1}
\end{equation}
Suppose that we have found $u$ with the appropriate regularity. Then,
\begin{equation}
v=\lambda u -f. \label{max2}
\end{equation}
To determine $z,$ recall that $z(.,0)=u_x,$ then, by $(\ref{max1})_3$, we obtain
\begin{equation}
z(.,\rho)=e^{-\lambda \tau \rho}u_x+ \tau e^{-\lambda \tau \rho}\int_0^\rho p(s)e^{\lambda \tau s}ds,\label{max22}
\end{equation}
and, in particular
\begin{equation}
z(x,1)=e^{-\lambda \tau}u_x+ z_0, \label{max3}
\end{equation}
with $z_0 \in L^2(\Omega)$ defined by
\begin{equation*}
z_0=\tau e^{-\lambda \tau}\int_0^1 p(s)e^{\lambda \tau s}ds. \label{max4}
\end{equation*}
Now, Multiplying $(\ref{max1})_2$ and $(\ref{max1})_4$ respectively by $w \in H^1_0(\Omega)$ and $\varphi \in H^2(\Omega)$ such that $\varphi_x(0)=\varphi_x(\ell)=0,$ we obtain after some integrations by parts taking into account boundary conditions on $v,$ $\theta$ and $w,$
\begin{equation}
\lambda \int_\Omega vw dx+ \int_\Omega \left(\alpha z(.,1)+\beta v_x \right)w_x dx + \gamma \int_\Omega \theta_x w dx = \int_\Omega gw dx \label{max5}
\end{equation}
and
\begin{equation}
\lambda \int_\Omega \theta \varphi dx- \gamma \int_\Omega v \varphi_x dx + \kappa \int_\Omega \theta_x \varphi_x dx = \int_\Omega h \varphi dx. \label{max6}
\end{equation}
Substituting (\ref{max2}) and (\ref{max3}) into (\ref{max5}) and (\ref{max6}), we get 
\begin{equation}
\lambda^2 \int_\Omega uw dx+\left(\alpha e^{-\lambda \tau}+\lambda \beta \right)  \int_\Omega u_x w_x dx + \gamma \int_\Omega \theta_x w dx = \int_\Omega (g+\lambda f)w dx + \int_\Omega (f_x-\alpha z_0)w_x dx \label{max7}
\end{equation}
and
\begin{equation}
\lambda \int_\Omega \theta \varphi dx- \lambda \gamma  \int_\Omega u \varphi_x dx + \kappa \int_\Omega \theta_x \varphi_x dx = \int_\Omega (h-\gamma f) \varphi dx. \label{max8}
\end{equation}
Summing (\ref{max7}), and (\ref{max8}) multiplied by $\frac{1}{\lambda},$  we get
\begin{equation}
b\left((u,\theta),(w,\varphi) \right)=F(w,\varphi) \label{max9}
\end{equation}
with
$$b\left((u,\theta),(w,\varphi) \right)= \int_\Omega \left[ \lambda^2uw+\left(\alpha e^{-\lambda \tau}+\lambda \beta \right) u_x w_x\right]  dx+\int_\Omega \left( \theta \varphi + \frac{\kappa}{\lambda}\theta_x \varphi_x\right)  dx+ \gamma \int_\Omega \left( \theta_x w- u \varphi_x\right) dx$$
and
$$F(w,\varphi)= \int_\Omega (g+\lambda f)w dx + \int_\Omega (f_x-\alpha z_0)w_x dx+\frac{1}{\lambda}\int_\Omega (h-\gamma f) \varphi dx.$$
The space $$\mathcal{F}:=\left\lbrace (w,\varphi) \in H^1_0(\Omega)\times H^2(\Omega)\mid \theta_x(0)=\theta_x(\ell)=0 \right\rbrace,$$
equipped with the inner product
$$\langle (w_1,\varphi_1),(w_2,\varphi_2) \rangle _\mathcal{F}=\int_\Omega \left(w_1w_2+w_{1x}w_{2x}+\varphi_1\varphi_2+\varphi_{1x}\varphi_{2x} \right)dx, $$
is a Hilbert space; the bilinear form $b$ on $\mathcal{F}\times \mathcal{F}$ and the linear form $F$ on $\mathcal{F}$ are continuous. Moreover, for every $(w,\varphi) \in \mathcal{F},$
$$\vert b\left((w,\varphi),(w,\varphi)\right) \vert \geq c\Vert (w,\varphi)\Vert ^2_\mathcal{H}$$
with $c:=\min \left(\lambda^2, \left(\alpha e^{-\lambda \tau}+\lambda \beta \right), 1, \frac{\kappa}{\lambda}\right)>0. $ 
 
 By the Lax-Milgram lemma, equation (\ref{max9}) has a unique solution $(u,\theta ) \in \mathcal{F}.$ Immediately, from (\ref{max2}), we have that $v \in H^1_0(\Omega).$ Now, if we consider $(w,\varphi) \in \lbrace 0\rbrace \times \mathcal{D}(\Omega)$ in (\ref{max9}) we deduce that  equation $(\ref{max1})_4$ holds true. The function $z,$ defined by (\ref{max3}), belongs to $L^1(\Omega , H^1(0,1))$ and satisfies $(\ref{max1})_3$ and $z(.,0)=u_x.$ 
 
The functions $z(.,1)$ and $v_x$ belong to $L^2(\Omega),$ then we take $(w,\varphi) \in \mathcal{D}(\Omega) \times \lbrace 0\rbrace$ in (\ref{max9}) to deduce that $\alpha z(.,1)+\beta v_x$ belongs to $H^1_0(\Omega)$ and that  equation $(\ref{max1})_1$ holds true. 
 
 Let $\tilde{\theta} = \theta-\frac{1}{\ell}\int_\Omega \theta_0dx,$ then  we have that  $U=(u,v,z,\tilde{\theta})$ belongs to $\mathcal{D}(\mathcal{A})$, and $\mathcal{A}U=(f,g,p,h).$ Thus, $\lambda Id-\mathcal{A}$ is surjective for every $\lambda>0$.
 \end{proof}
  In conclusion the operator $\mathcal{A}-mId$ generates a $\mathcal{C}_0$-semigroup of contraction. By the bounded perturbation theorem (Sect. III.1 of \cite{eng99}), we have
 \begin{lemma}
 The operator $\mathcal{A}$ generates a $\mathcal{C}_0$-semigroup on $\mathcal{H}$.
 \end{lemma}
  Finally, the well-posedness result follows from semigroup theory.
  \begin{theorem}
For any initial datum $U_0 \in \mathcal{H}$ there exists a unique solution $U \in \mathcal{C}([0,+\infty),\mathcal{H})$ of problem (\ref{d500}). Moreover, if $U_0 \in \mathcal{D}(\mathcal{A}),$ then $U \in \mathcal{C}([0,+\infty),\mathcal{D}(\mathcal{A}))\cap\mathcal{C}^1([0,+\infty),\mathcal{H}).$
\end{theorem}

\section{Exponential stability}
Based on Lyapunov method, we prove that the system (\ref{d1.0}) is exponentially stable for some $\beta>0.$ More precisely:
\begin{theorem}
There exists $\beta_0>0$ such that for every $\beta \geq \beta_0,$ the system (\ref{d1.0}) is exponentially stable.
\end{theorem}
\begin{proof}
We take as Lyapunov function
\begin{equation*}
V(t):=N_1V_1(t)+\alpha N_2V_2(t)+N_3V_3(t)+N_4V_4(t)+N_5V_5(t)+N_6V_6(t)
\end{equation*} 
where
\begin{eqnarray*}
V_1(t)&:=&\frac{1}{2}\left\| u_t\right\|^2=\frac{1}{2}\int_\Omega u_t^2dx, \;\;V_2(t):=\;\frac{1}{2}\left\| u_x\right\|^2=\frac{1}{2}\int_\Omega u_x^2dx,\;\;
V_3(t):=\;\frac{1}{2}\left\| \theta\right\|^2=\frac{1}{2}\int_\Omega \theta^2dx, \\
V_4(t)&:=&\int_0^1e^{-2\lambda \rho}\left\| z(.,\rho,t)\right\|^2d\rho =\int_0^1e^{-2\lambda \rho}\int_\Omega z^2(x,\rho,t)dx d\rho,\\
V_5(t)&:=&-\int_0^1e^{-\lambda \rho}f(\rho)\left\langle z(.,\rho,t),u_x\right\rangle d\rho =-\int_0^1e^{-\lambda \rho}f(\rho)\int_\Omega z(x,\rho, t)u_x(x,t)dx d\rho\;\;\text{and}\\
V_6(t)&:=&\langle u,u_t \rangle =\int_\Omega uu_t dx.
\end{eqnarray*}
 $f$ is a real function defined on $[0,1]$ and that will be determined later. The constants $N_1, N_2, N_3, N_4, N_5$ and $ N_6$ are positive numbers to be fixed later too.

Denote by $\tilde{V}(t)$  the energy defined by
$$\tilde{V}(t):=N_1V_1(t)+\alpha N_2V_2(t)+N_3V_3(t)+N_4V_4(t).$$
It is clear that $\tilde{V}(t)$ is equivalent to $E(t).$ Then for a suitable choice of  $f$ we will prove that we can find $ \lbrace N_1,...,N_6\rbrace$ and $\beta>0$ such that the following two assumptions are satisfied:
\begin{enumerate}
\item[(A1)]  $V(t)$ is equivalent to $\tilde{V}(t),$
\item[(A2)] $V^\prime(t) \leq -n_0\tilde{V}(t),$
for some positive number $n_0$. 
\end{enumerate}
The rest of the proof will be divided into three parts:

\textit{First part:} it concerns the second assumption (A2). We start with the following lemma
\begin{lemma} \label{lem2}
Let $V(t)$ be defined as before. By choosing a function $f$ satisfying
\begin{equation}
-e^{-2\lambda \rho}=(e^{-\lambda \rho} f(\rho))^\prime,\;\;\lambda>0 \label{cf}
\end{equation}
and by taking $N_3=N_1,$ $N_6\beta=N_2\alpha$ and $N_6\alpha=\frac{f(1)e^{-\lambda}}{\tau}N_5$ we have that for every positive real numbers $\varepsilon_1, \varepsilon_2, \varepsilon_3$ and $\varepsilon_4,$ 
\begin{eqnarray*}
V^\prime (t)
 &\leq & \left( -N_4 \frac{e^{-2\lambda}}{\tau}+N_1\frac{\alpha \varepsilon_1}{2} \right)\left\| z(.,1,.)\right\|^2 \\
 &+&\left(-2\frac{k_1}{\tau}N_4+\frac{N_5}{2\tau}\left( \varepsilon_2+\frac{\tau}{\varepsilon_3}\right)   \right)V_4(t) \\
&+& \left(  \frac{N_4}{\tau}+\frac{N_5}{\tau}\left(\frac{\Gamma}{2\varepsilon_2}-\Lambda \right)+N_5\frac{\Psi\gamma \varepsilon_4 c_p}{2\alpha \tau} \right)\left\| u_x\right\|^2 \\
&+& \left(N_1\left(\frac{\alpha}{2\varepsilon_1}-\beta \right) +N_5\left( \frac{\varepsilon_3}{2}\Phi +\frac{\Psi c_p}{2\alpha \tau}\right)   \right)\left\| u_{tx}\right\|^2\\
&+&\left( -N_1\kappa+N_5\frac{\Psi\gamma }{2\alpha \tau\varepsilon_4} \right) \left\| \theta_{x}\right\|^2 
\end{eqnarray*}
\end{lemma}
where $c_p>0$ is the Poincaré constant associated to $\Omega,$ (it can be taken equal to $\frac{\ell^2}{2}$) and
\begin{equation*}
\Psi :=f(1)e^{-\lambda},\;\;\;\Lambda :=f(0)\;\;\;\text{and}\;\;\Phi:=\int_0^1f^2(\rho)d\rho.
\end{equation*}
Furthermore, $$\Gamma:=\int_0^1e^{-2\lambda \rho}d\rho=\frac{1-e^{-2\lambda}}{2\lambda}.$$
Notice that, in view of (\ref{cf}), we have
\begin{equation*}
\Lambda =\Psi+\Gamma.
\end{equation*}
\begin{proof}
Computing the derivatives of $V_1, V_2, V_3, V_4$ and $V_5$
using integration by parts, boundary conditions and Youg's inequality, we have
\begin{eqnarray*}
V^\prime _1(t)=\left\langle u_{tt},u_t\right\rangle &=&\left\langle (\alpha z(.,1,.)+\beta u_{tx})_x,u_t\right\rangle - \gamma \left\langle \theta_x,u_t\right\rangle\\
&=&-\left\langle \alpha z(.,1,.)+\beta u_{tx},u_{tx}\right\rangle - \gamma \left\langle \theta_x,u_t\right\rangle \\
&=&-\left\langle \alpha z(.,1,.),u_{tx}\right\rangle -\beta \left\| u_{tx}\right\|^2  - \gamma \left\langle \theta_x,u_t\right\rangle \\
&\leq & \left(\frac{\alpha}{2\varepsilon_1}-\beta \right)\left\| u_{tx}\right\|^2 +\frac{\alpha \varepsilon_1}{2}\left\| z(.,1,.)\right\|^2- \gamma \left\langle \theta_x,u_t\right\rangle,
\end{eqnarray*} 
\begin{eqnarray*}
V^\prime _2(t)&=& \left\langle u_{xt},u_x\right\rangle, 
\end{eqnarray*}
\begin{eqnarray*}
V^\prime _3(t)=\left\langle \theta_t,\theta\right\rangle &= &-\gamma \left\langle (u_{xt},\theta\right\rangle + \kappa\left\langle \theta_{xx},\theta\right\rangle
\\ 
&=& \gamma \left\langle u_{t},\theta_x\right\rangle - \kappa \left\| \theta_{x}\right\|^2.
\end{eqnarray*}
The derivative of $V_4$ is
\begin{eqnarray*}
V^\prime _4(t)&= &2\int_0^1e^{-2\lambda \rho}\left\langle z(.,\rho,.),z_t(.,\rho,.)\right\rangle d\rho \\ &= &-\frac{2}{\tau}\int_0^1e^{-2\lambda \rho}\left\langle z(.,\rho,.),z_\rho(.,\rho,.)\right\rangle d\rho \\
&= &-\frac{e^{-2\lambda}}{\tau}\left\| z(.,1,.)\right\|^2+\frac{1}{\tau}\left\| u_{x}\right\|^2-\frac{2\lambda}{\tau}\int_0^1e^{-2\lambda \rho}\left\| z(.,\rho,.)\right\|^2d\rho \\
&\leq &-\frac{e^{-2\lambda}}{\tau}\left\| z(.,1,.)\right\|^2+\frac{1}{\tau}\left\| u_{x}\right\|^2-\frac{2\lambda}{\tau}V_4(t).
\end{eqnarray*}
The derivative of $V_5$ is calculated as follows
\begin{eqnarray*}
V^\prime _5(t)&= &\frac{1}{\tau}\int_0^1e^{-\lambda \rho}f(\rho)\left\langle z_\rho(.,\rho,.),u_x\right\rangle d\rho - \int_0^1e^{-\lambda \rho}f(\rho)\left\langle z(.,\rho,.),u_{xt}\right\rangle d\rho \\
&= &\frac{1}{\tau}e^{-\lambda}f(1)\left\langle z(.,1,.),u_x\right\rangle-\frac{1}{\tau}f(0)\left\| u_{x}\right\|^2\\
 &+ &\frac{1}{\tau}\int_0^1(e^{-\lambda \rho}f(\rho))^\prime\left\langle z(.,\rho,.),u_{x}\right\rangle d\rho - \int_0^1e^{-\lambda \rho}f(\rho)\left\langle z(.,\rho,.),u_{xt}\right\rangle d\rho.
\end{eqnarray*}
Replacing $e^{-\lambda}f(1)$ by $\Psi,$ $f(0)$ by $\Lambda$ and $(e^{-\lambda \rho}f(\rho))^\prime$ by $-e^{-2\lambda \rho},$ we obtain (using Young's inequality),
\begin{eqnarray*}
V^\prime _5(t)&\leq & \frac{1}{\tau}\left(\frac{\Gamma}{2\varepsilon_2}-\Lambda \right)\left\| u_{x}\right\|^2+\frac{1}{2\tau}\left( \varepsilon_2+\frac{\tau}{\varepsilon_3}\right)V_4(t)+\frac{\varepsilon_3}{2}\Phi \left\| u_{tx}\right\|^2+\frac{\Psi}{\tau}\left\langle z(.,1,.),u_x\right\rangle . 
\end{eqnarray*}
Finally, the derivative of $V_6$ is
\begin{eqnarray*}
V^\prime _6(t)&= & \left\| u_{t}\right\|^2-\alpha \left\langle z(.,1,.),u_x\right\rangle-\beta \left\langle u_x,u_{xt}\right\rangle +\gamma \left\langle u,\theta_x\right\rangle.
\end{eqnarray*}
To conclude, it suffices to sum up $N_1V^\prime_1(t),$ $\alpha N_2V^\prime_2(t),$ $N_3V^\prime_3(t),$ $N_4V^\prime_4(t),$ $N_5V^\prime_5(t),$ and $N_6V^\prime_6(t).$
\end{proof}
In view of Lemma \ref{lem2}, for the assumption (A2) to be satisfied, it suffices that
\begin{eqnarray} 
   -N_4 \frac{e^{-2\lambda}}{\tau}+N_1\frac{\alpha \varepsilon_1}{2} &= &0, \label{n1} \\
n_1:=-2\frac{\lambda}{\tau}N_4+\frac{N_5}{2\tau}\left( \varepsilon_2+\frac{\tau}{\varepsilon_3}\right) &< &0, \label{n2}\\
 n_2:= \frac{N_4}{\tau}+\frac{N_5}{\tau}\left(\frac{\Gamma}{2\varepsilon_2}-\Lambda \right)+N_5\frac{\Psi\gamma \varepsilon_4 c_p}{2\alpha \tau}
  &< &0, \label{n3}\\
n_3:=N_1\left(\frac{\alpha}{2\varepsilon_1}-\beta \right) +N_5\left( \frac{\varepsilon_3}{2}\Phi +\frac{\Psi c_p}{\alpha \tau}\right) &< &0, \label{n4}\\
n_4:= -N_1\kappa+N_5\frac{\Psi\gamma }{2\alpha \tau\varepsilon_4}  &< &0.\label{n5} 
\end{eqnarray} 
The first condition (\ref{n1}) is equivalent to
\begin{equation*}
N_4=aN_1
\end{equation*}
with $a:=\frac{1}{2}\alpha\varepsilon_1\tau e^{2\lambda}.$ \label{n4_1}

The second condition (\ref{n2}) means that there exists $0<k<1$ such that
\begin{equation*}
N_5=bN_4=abN_1 \label{n5_1}
\end{equation*}
with $b:=\frac{4\lambda k}{\varepsilon_2+\frac{\tau}{\varepsilon_3}}.$

Note that, we have then
\begin{equation*}
N_6=\frac{\Psi}{\alpha \tau}N_5=\frac{ab\Psi}{\alpha \tau}N_1\;\;\;\text{and}\;\;\;N_2=\frac{\beta}{\alpha}N_6=\frac{ab\Psi\beta}{\alpha^2 \tau}N_1. \label{n6_1}
\end{equation*}

Replacing $N_5$ by $abN_1$ and $a$ by $\frac{1}{2}\alpha\varepsilon_1\tau e^{2\lambda}$ in (\ref{n4}), then multiplying the inequality by $\frac{\alpha}{\varepsilon_1}$, we obtain
\begin{equation}
\frac{1}{2}b\Psi\alpha e^{2\lambda}c_p+\frac{1}{4}\tau b \Phi \varepsilon_3 \alpha^2 e^{2\lambda}<\frac{\alpha}{\varepsilon_1}\left(\beta-\frac{\alpha}{2\varepsilon_1} \right). \label{n44} 
\end{equation}
We take $\varepsilon_1=\frac{\alpha}{\beta},$ then (\ref{n44}) turns into
\begin{equation}
b\Psi\alpha e^{2\lambda}c_p+\frac{1}{2}\tau b \Phi \varepsilon_3 \alpha^2 e^{2\lambda}<\beta^2. \label{eqfond2}
\end{equation}
Return back to (\ref{n5}), replacing $N_5$ by $ab
N_1$ to obtain
\begin{equation}
\frac{ab\Psi\gamma}{2\alpha \tau \varepsilon_4}<\kappa. \label{n5'}
\end{equation}
Also inequality (\ref{n3}) becomes
\begin{equation}
\frac{b\Psi\gamma \varepsilon_4 c_p}{2\alpha } <\left( \left( \Lambda-\frac{\Gamma}{2\varepsilon_2}\right)b-1 \right). \label{n55}
\end{equation}
Already, it is necessary that $\Lambda-\frac{\Gamma}{2\varepsilon_2}>0,$ that is $\Lambda>\frac{\Gamma}{2\varepsilon_2},$ and $\left( \Lambda-\frac{\Gamma}{2\varepsilon_2}\right)b-1>0,$ that is,
\begin{equation}
b=\frac{A}{\Lambda-\frac{\Gamma}{2\varepsilon_2}},\;\;\;A>1, \label{defb}
\end{equation}
hence, (\ref{n55}) turns into
\begin{equation}
\frac{b\Psi\gamma \varepsilon_4 c_p}{2\alpha } <\left( A-1 \right). \label{n33}
\end{equation}
Combining (\ref{n5'}) and (\ref{n33}) to obtain
\begin{equation}
\frac{ab\Psi\gamma}{2 \alpha \tau \kappa}<\varepsilon_4<\frac{2\alpha}{\gamma b\Psi c_p}(A-1)\label{eqfond1'}.
\end{equation}
Replacing $a$ by $\frac{1}{2\beta}\alpha^2\tau e^{2\lambda}$ in (\ref{eqfond1'}) to get
 \begin{equation}
\frac{\alpha \gamma b\Psi}{4\beta \kappa}e^{2\lambda}<\varepsilon_4<\frac{2\alpha}{\gamma b\Psi c_p}(A-1) \label{eqfond1}.
\end{equation}
Now, going back with more detail on assumption (\ref{defb}). To do this, replacing $b$ by $\frac{4\lambda k}{\varepsilon_2+\frac{\tau}{\varepsilon_3}},$ we obtain
\begin{equation}
\frac{A}{\Lambda-\frac{\Gamma}{2\varepsilon_2}}=\frac{4\lambda k}{\varepsilon_2+\frac{\tau}{\varepsilon_3}}, 
\end{equation}
or equivalently,
\begin{equation}
\frac{A}{4\lambda k}\varepsilon_2^2-\left(\Lambda-\frac{A}{4\lambda k}\frac{\tau}{\varepsilon_3} \right)\varepsilon_2+\frac{\Gamma}{2}=0 
\end{equation}
the discriminant of such equation in $\varepsilon_2$ is
\begin{equation}
\Delta:=\left(\Lambda-\frac{A}{4\lambda k}\frac{\tau}{\varepsilon_3} \right)^2-\frac{A\Gamma}{2\lambda k}
\end{equation}
which must be at least zero. In the sequel, we choose it zero. On the other hand $\varepsilon_2$ is positive, then
\begin{equation}
\Lambda-\frac{A}{4\lambda k}\frac{\tau}{\varepsilon_3}=\sqrt{\frac{A\Gamma}{2\lambda k}}
\end{equation}
or equivalently
\begin{equation}
\frac{A}{4\lambda k}\frac{\tau}{\varepsilon_3}=\Lambda-\sqrt{\frac{A\Gamma}{2\lambda k}}. \label{dis0}
\end{equation}
It is obvious that the left hand side of the last equation is positive, that is
\begin{equation}
\frac{A}{k}<2\lambda\frac{\Lambda^2}{\Gamma}. \label{dis}
\end{equation}
Moreover, since $A>1$ and $0<k<1$ we have
\begin{equation}
1<\frac{A}{k}<2\lambda\frac{\Lambda^2}{\Gamma}. \label{dis2}
\end{equation} 
Finally, note that
\begin{equation}
\varepsilon_2=\sqrt{2\lambda \Gamma \frac{k}{A}}. \label{ep3}
\end{equation}
\textit{Second part:} it concerns the equivalence between  $V(t)$ and $\tilde{V}(t).$  Let $\varepsilon_5>0$ and $\varepsilon_6>0,$ we have
$$\left| N_5V_5\right|\leq \frac{N_5}{2\varepsilon_5}V_4+\frac{N_5 \Phi \varepsilon_5}{2} \left\|u_{x}\right\|^2$$
and
$$\left|N_6V_6 \right|\leq \frac{N_6\varepsilon_6}{2}c_p\left\|u_{x}\right\|^2+\frac{N_6}{2\varepsilon_6} \left\|u_{t}\right\|^2.$$
For $V(t)$ to be equivalent to $\tilde{V} (t)$ it is sufficient that
\begin{eqnarray}
\frac{N_6}{2\varepsilon_6}<\frac{N_1}{2},\;\;\;\;\;\;\frac{N_5}{2\varepsilon_5}<N_4, \label{ep67}\\
\frac{N_6\varepsilon_6}{2}c_p+\frac{N_5 \Phi \varepsilon_5}{2}<\frac{N_2\alpha}{2}. \label{ep67'}
\end{eqnarray}
Using $N_6=\frac{ab\Psi}{\alpha \tau}N_5$ and $N_5=abN_1$ in(\ref{ep67}) we get
\begin{equation*}
\varepsilon_6>\frac{ab\Psi}{\tau \alpha},\;\;\;\;\;\;\varepsilon_5>\frac{b}{2}.
\end{equation*}
We choose
\begin{equation*}
\varepsilon_6=2\frac{ab\Psi}{\tau \alpha},\;\;\;\;\;\;\varepsilon_5=b.
\end{equation*}
Using again $N_6=\frac{\Psi}{\alpha \tau}N_5$ and $N_2=\frac{\beta \Psi}{\alpha^2 \tau}N_5,$ inequality (\ref{ep67'}) becomes
\begin{equation}
\frac{b\Psi^2}{\beta \tau}c_p+ \Phi b<\frac{\beta \Psi}{\alpha \tau}. \label{eqfond3}
\end{equation}
\textit{Third part:} It is enough to examine the equations (\ref{dis2}), (\ref{eqfond2}),  (\ref{eqfond1}) and (\ref{eqfond3}).

We take $h:=\frac{\Psi}{\Gamma},$ then $\Lambda=\Psi+\Gamma=(1+h)\Gamma=(1+h)\frac{1-e^{-2\lambda}}{2\lambda}$.

\textit{First step.} We begin by assumption (\ref{dis2}) which can be translated into
\begin{equation}
1<\frac{A}{k}<(1-e^{-2\lambda})(1+h)^2 \label{eqfond0}
\end{equation}
We choose  $h:=e^{-2\lambda}.$ Then 
\begin{equation*}
(1-e^{-2\lambda})(1+h)^2=(1-h)(1+h)^2=1+h-h^2-h^3>1
\end{equation*}
for $\lambda$ large enough. We choose $A=1+h-h^2-2h^3-4h^4$ and $k=1-h^4.$ We have, for $\lambda$ large enough, $A>1$, $0<k<1$ and (\ref{eqfond0}) is satisfied since
\begin{equation}
1<\frac{A}{k}=1+h-h^2-2h^3+o(h^3)<(1-h)(1+h)^2
\end{equation}
\textit{Second step.} Estimate of $\varepsilon_2,\;b$ and $\varepsilon_3$ according to $h$ and $\lambda$ for $\lambda$ large enough. 
We have
\begin{eqnarray}
\frac{k}{A}=\frac{1+o(h^2)}{1+h-h^2+o(h^2)} \notag\\
=1-h+2h^2+o(h^2)
\end{eqnarray}
and 
\begin{equation*}
\Gamma =\frac{1}{2\lambda}(1-h)
\end{equation*}
then
\begin{eqnarray*}
2\lambda\Gamma\frac{k}{A}&=&(1-h)(1-h+2h^2+o(h^2))\\
&=&1-2h+3h^2+o(h^2).
\end{eqnarray*}
Hence we obtain, using (\ref{ep3}),
\begin{eqnarray*}
\varepsilon_2 &=& 1-h+\frac{3}{2}h^2-\frac{1}{2}h^2+o(h^2)\\
&=&1-h+h^2+o(h^2).
\end{eqnarray*}
We evaluate $b$. First,
\begin{eqnarray*}
\frac{1}{2\varepsilon_2} &=&\frac{1}{2(1-h+h^2+o(h^2))}\\
&=&\frac{1}{2}\left( 1+h-h^2+(h-h^2)^2+o(h^2)\right) \\
&=&\frac{1}{2}\left( 1+h+o(h^2)\right), 
\end{eqnarray*}
then
\begin{eqnarray*}
1+h-\frac{1}{2\varepsilon_2}=\frac{1}{2} (1+h+o(h^2)),
\end{eqnarray*}
hence
\begin{eqnarray*}
\frac{1}{1+h-\frac{1}{2\varepsilon_2}}&=&2 (1-h+h^2+o(h^2)).
\end{eqnarray*}
Finally, from(\ref{defb}) and using that $\Lambda=(1+h)\Gamma$ we have
\begin{eqnarray}
b=\frac{ A}{\Gamma(1+h-\frac{1}{2\varepsilon_2})}&=&4\lambda\frac{(1+h-h^2+o(h^2))(1-h+h^2+o(h^2))}{1-h} \notag\\
&=&4\lambda (1+h+o(h^2)).
\end{eqnarray}
Now we evaluate $\varepsilon_3:$

First, recall that
\begin{equation}
\frac{A}{2\lambda k}=\frac{1}{2\lambda}(1+h-h^2-2h^3-3h^4+o(h^4)), \label{ep4}
\end{equation}
 then $$\frac{A\Gamma}{2\lambda k}=\frac{1}{4\lambda^2}(1-2h^2-h^3-h^4+o(h^4))$$ and 
\begin{equation}
\sqrt{\frac{A\Gamma}{2\lambda k}}=\frac{1}{2\lambda}(1-h^2-\frac{1}{2} h^3-h^4+o(h^4)) \label{ep4'}
\end{equation}
Using (\ref{dis0}), (\ref{ep4}), (\ref{ep4'}) and that $\Lambda =\frac{1}{2\lambda}(1-h^2)$, we have
\begin{equation}
\varepsilon_3=\frac{\tau}{h^3}(1-h+o(h)).
\end{equation} 
\textit{Third step.} Interpretation of
Inequality (\ref{eqfond2}). First, we need to express $\Phi$ according to $h$ and $\lambda.$ 

Since
\begin{eqnarray*}
f(\rho)&=&e^{\lambda \rho}\left(h\Gamma+\int_\rho^1 e^{-2\lambda s}ds\right)\\
&=& \frac{1}{2\lambda}e^{\lambda\rho}\left(e^{-2\lambda}(1-e^{-2\lambda})+(e^{-2\lambda\rho}-e^{-2\lambda}) \right)\\
&=& \frac{1}{2\lambda}e^{\lambda\rho}\left(e^{-2\lambda \rho}-e^{-4\lambda} \right)
\end{eqnarray*}
then,
\begin{eqnarray*}
f^2(\rho)&=&\frac{1}{4\lambda^2}e^{2\lambda \rho}\left(e^{-4\lambda\rho }-2e^{-4\lambda}e^{-2\lambda \rho}+e^{-8\lambda} \right)\\
&=&\frac{1}{4\lambda^2}\left(e^{-2\lambda\rho }-2e^{-4\lambda}+e^{-8\lambda}e^{2\lambda\rho } \right).
\end{eqnarray*}
Hence,
\begin{eqnarray*}
\Phi &=&\int_0^1 f^2(\rho )d\rho \\
 &=&\frac{1}{4\lambda^2}\left( \frac{1-e^{-2\lambda}}{2\lambda}-2e^{-4\lambda}+\frac{1}{2\lambda}\left( e^{-6\lambda}-e^{-8\lambda}\right)  \right)\\ &=& \frac{1}{8 \lambda^3}(1-h-4\lambda h^2+o(h^2)). 
\end{eqnarray*}
Now, inequality (\ref{eqfond2}) can be rewritten as:
\begin{equation*}
 2\alpha c_p(1-h^2+o(h^2))+\frac{\tau^2\alpha^2}{4\lambda^2h^4}(1-h+o(h))<\beta^2. 
\end{equation*}
Then, we take 
\begin{equation}
2 \left( \frac{c_p}{\alpha \tau^2}+\frac{e^{8\lambda}}{8\lambda^2}\right)<\left( \frac{\beta}{\alpha \tau}\right)^2  \label{cond1}  
\end{equation}
with $\lambda$ large enough.

\textit{Fourth step.} Condition (\ref{eqfond1}) and existence of $\varepsilon_4.$
Inequality (\ref{eqfond1}) can be rewritten as:
\begin{equation*}
\frac{\alpha \gamma}{2\beta \kappa}(1-h^2+o(h^2))<\varepsilon_4<\frac{\alpha}{\gamma c_p}(1-h+o(h)).
\end{equation*}
Then, we take
\begin{equation}
\frac{\gamma^2}{2\kappa}c_p<\beta (1-h+o(h))  \label{cond2} 
\end{equation}
and $\varepsilon_4$ can be taken equal to $\frac{\alpha}{2}\left( \frac{\gamma }{2\beta \kappa}+\frac{1}{\gamma c_p}(1-h+o(h))\right),$ with $\lambda$ large enough. 

\textit{Fifth step.} Interpretation of assumption (\ref{eqfond3}). It can be rewritten as:
\begin{equation*}
2h\frac{c_p}{\tau \beta}(1-h^2+o(h^2))+\frac{1}{\lambda h}(1+h+o(h))<\frac{\beta}{\tau \alpha}. 
\end{equation*}
It suffices to take
\begin{equation}
\left(2e^{-2\lambda}\frac{c_p}{\tau \beta}+\frac{1}{\lambda}(e^{2\lambda}+1+o(1)) \right)<\frac{\beta}{\alpha \tau }. \label{cond3} 
\end{equation}
with $\lambda$ large enough.

Note that for $\lambda$ large enough, $\beta =\alpha \tau e^{4\lambda}$ satisfies the three conditions (\ref{cond1}), (\ref{cond2}) and (\ref{cond3}). Moreover, there exists $\beta_0>0$ such that every $\beta>\beta_0$ satisfies the three conditions (\ref{cond1}), (\ref{cond2}) and (\ref{cond3}).

For every $\beta>\beta_0$ we have
\begin{equation}
\dot{V}(t)\leq -n_0 \tilde{V}(t),
\end{equation} 
where $n_0=min\lbrace n_1,n_2,\frac{1}{c_p} n_3,\frac{1}{c_p}n_4 \rbrace.$
Recall that $V(t)$, $\tilde{V}(t)$ and $E(t)$ are equivalent then, there exists $a_0>0, C>0$ such that
$$E(t)<Ce^{-a_0t}.$$
\end{proof}

\section*{Comments}
 We can replace the Neumann conditions for $\theta$
$$\theta_x(0,t)=\theta_x(\ell,t)=0$$ by the Dirichlet conditions
 $$\theta(0,t)=\theta(\ell,t)=0,$$ we then obtain the same results.

\bibliographystyle{amsplain}
\bibliography{mabib_2art3}

\end{document}